\documentclass[12pt]{amsart}
\usepackage{amssymb,mathrsfs}
\usepackage[matrix,arrow,curve]{xy}
\usepackage{epsfig}
\usepackage{setspace}
\usepackage{url}
\newcounter{cequation}[section]

\newtheorem{theorem}[cequation]{Theorem}
\newtheorem*{theorem*}{Theorem}
\newtheorem{lemma}[cequation]{Lemma}

\newtheorem{proposition}[cequation]{Proposition}

\newtheorem{problem}[cequation]{Problem}
\newtheorem{lemma-definition}[cequation]{Lemma-Definition}

\theoremstyle{definition}
\newtheorem{example}[cequation]{Example}
\newtheorem{definition}[cequation]{Definition}
\newtheorem*{definition*}{Definition}

\theoremstyle{remark}
\newtheorem{remark}[cequation]{Remark}

\makeatletter\@addtoreset{equation}{section}
\@addtoreset{cequation}{section}

\makeatother

\def \G {\mathcal{G}}
\def \X {\mathcal{X}}
\def \Y {\mathcal{Y}}

\def \O {\mathcal{O}}

\def \QQ {\mathcal{Q}}

\def \R {\mathbb{R}}

\def \C {\mathbb{C}}
\def \CC {\mathbb{C}}
\def \TT {\mathbb{T}}

\def \ZZ {\mathbb{Z}}

\def \Ar {\mathrm{Ar}}

\def \Ver {\mathrm{Ver}}

\def \wt {\mathrm{wt}}

\def \vv {\mathbf{v}}
\def \hh {\mathbf{h}}

\def \HB {\mathrm{HB}}
\def \MB {\mathrm{MB}}
\def \VB {\mathrm{VB}}

\def \G {\mathrm{Gr}}

\newcommand{\Spec}{\mathrm{Spec}\,}

\newcommand{\fund}{\mathbf{1}}

\newcommand{\arrow}[2]{\langle #1\to #2 \rangle}

\def \ge {\geqslant}
\def \le {\leqslant}

\title{Laurent phenomenon for Landau--Ginzburg models of complete intersections in Grassmannians}

\author{Victor Przyjalkowski, Constantin Shramov}

\thanks{
This work was performed in Steklov Mathematical Institute and supported by the Russian Science Foundation under grant 14-50-00005.
}

\address{
Steklov Mathematical Institute, 8 Gubkina st., Moscow, Russia, 119991\newline
\phantom{La}National Research University Higher School of Economics, Russian Federation, AG Laboratory, HSE, 7 Vavilova str., Moscow, Russia, 117312}\email{victorprz@mi.ras.ru, costya.shramov@gmail.com}

\begin{document}

\begin{abstract}
In 1997 Batyrev, Ciocan-Fontanine, Kim, and van Straten suggested a construction of Landau--Ginzburg models for Fano complete intersections
in Grassmannians similar to Givental's construction for complete intersections in smooth toric varieties.
We show that for a Fano complete intersection in Grassmannians 
the result of the above construction is birational to a complex torus.
In other words, the complete intersections under consideration have very weak Landau--Ginzburg models.
\end{abstract}

\maketitle

\section{Introduction}

There are several versions of Mirror symmetry conjectures depending on the data one wants to operate with.
One of them, Mirror symmetry of variations of Hodge structures (see, for example,~\cite[Conjecture~38]{Prz13}),
predicts that for any smooth Fano variety $X$ of dimension $N$
there exists
a \emph{very weak Landau--Ginzburg model} $f_X$, i.\,e. a Laurent polynomial in $N$ variables
satisfying a so called \emph{period condition}. Namely,
the periods of the family $\{f_X=\lambda\mid \lambda\in \CC\}$ should be solutions of a \emph{regularized quantum differential equation} for $X$,
that is a differential equation constructed in terms of Gromov--Witten invariants of $X$ (for a bit more details and references see Section~\ref{section:preliminaries}).

So far, there is no general construction of very weak Landau--Ginzburg model. However, there are several approaches.
One of them uses
toric degenerations of the Fano variety, see~\cite{Prz13},~\cite{Prz11},~\cite{ILP13},~\cite{CCGGK12}.
In~\cite{BCFKS97} (see also~\cite{BCFKS98}) the Givental's construction of Landau--Ginzburg models for complete
intersections in smooth toric varieties was generalized to complete intersections in Grassmannians. For this well known
toric degenerations of Grassmannians themselves were used.
Instead of a Laurent polynomial $f_X$, that is a regular function on a torus $(\CC^*)^{N}$, a certain complete intersection in $(\CC^{*})^{N_0}$, $N_0>N$,
with a regular function $F_{X}$ called \emph{superpotential} is constructed.
In~\cite[Problem~16]{Prz13} it was conjectured that the complete intersection constructed in~\cite{BCFKS97} is birational to a torus
$(\CC^*)^N$ and the function constructed in~\cite{BCFKS97} corresponds under this birationality to a Laurent polynomial which
satisfies the required period condition.
In~\cite{PSh14a} we solve this problem for complete intersections in Grassmannians of planes $\mathrm{Gr}(2,k+2)$. Moreover,
we give an algorithm to construct a corresponding birational map. 
The idea is to use a monomial change of variables, after which the Laurent polynomial representing $F_{X}$ has some chosen variables placed
in denominators only. This enables one to make non-toric change of variables keeping the Laurent form of the function.
It was expected that a similar algorithm can be constructed for complete intersections in any Grassmannian as well.
We do it in this paper. Our algorithm uses the same idea as the one we applied for complete intersections in
Grassmannians of planes, but it is of a more geometric nature.

In~\cite{DH15} it was shown by other methods that for a Fano complete
intersection in a Grassmannian the Landau--Ginzburg model constructed in~\cite{BCFKS97} is birational to a torus with a regular function on it, so it can be represented by Laurent polynomial; moreover, this polynomial
is recovered from a toric degeneration. The way of constructing such polynomial from~\cite{DH15} is indirect,
and it is unclear how to establish the period condition using it.
On the contrary, our algorithm is more adapted to the case of complete intersections in Grassmannians and thus it has a simple straightforward form that, in particular, enables one to check the period condition, so the polynomials it produces
are very weak Landau--Ginzburg models.

\medskip

The paper is organized as follows. In Section~\ref{section:preliminaries} we provide definitions and constructions
we need, in particular, the definition of toric Landau--Ginzburg models and the construction
of Givental's type Landau--Ginzburg models for complete intersections in Grassmannians presented in~\cite{BCFKS97}.
In Section~\ref{section: main theorem} we prove our main theorem
(Theorem~\ref{theorem: main}) stating that
the Landau--Ginzburg models are birational to complex tori $(\C^*)^N$, and provide a formula for their superpotentials.
In Section~\ref{section: integral} we check the period condition for them;
as a corollary we get an existence of very weak Landau--Ginzburg models for
Fano complete intersections in Grassmannians 

\medskip

We are grateful to A.\,Kuznetsov whose suggestions drastically improved the paper.

\medskip

{\bf Notation and conventions.} Everything is defined over $\C$.
Given two integers $n_1$ and $n_2$, we denote the set $\{i\in \ZZ\mid n_1\le i\le n_2\}$ by $[n_1,n_2]$.
When we speak about hyperplane or hypersurface sections of a Grassmannian we mean hyperplane or hypersurface sections in its Pl\"ucker embedding.

\section{Preliminaries}
\label{section:preliminaries}

\subsection{Toric Landau--Ginzburg models}
\label{subsection:toric LG}

Let $X$ be a smooth Fano variety of Picard number $1$ and dimension $N$.
Using Gromov--Witten
invariants of $X$ (see~\cite{Ma99}) one can define a series 
$$
\widetilde{I}^X_{0}=\left(1+ \sum_{d\ge 2} d!\langle\tau_{d-2}
\fund\rangle_d \cdot t^d\right)\in \CC[[t]]
$$
called a \emph{constant term of regularized $I$-series for $X$}, see for instance~\cite{Prz08b}. This series is a
solution of \emph{regularized quantum differential equation} given by the second Dubrovin's
connection. For more details one can look at~\cite{Gi96},~\cite{Pa98},~\cite{Prz08b}, and~\cite{GS07}.
In~\cite{BCFKS97} one can find a way to write down the formula for this series when $X$ is a complete intersections in a Grassmannian;
we recover it in Theorem~\ref{theorem: periods of CI}.

Denote a constant term of a Laurent polynomial $f$ by $[f]$.
\begin{definition}[{see~\cite[\S6]{Prz13}}]
\label{definition: toric LG}
\emph{A toric Landau--Ginzburg model} of $X$ is a Laurent polynomial $\mbox{$f_X\in \CC[x_1^{\pm 1}, \ldots, x_N^{\pm 1}]$}$ which satisfies:
\begin{description}
  \item[Period condition] One has $\sum [f_X^i]t^i=\widetilde{I}^X_0$.
  \item[Calabi--Yau condition] There exists a fiberwise compactification of a family
$$f_X\colon (\CC^*)^N\to \CC$$
whose total space is a (non-compact) smooth Calabi--Yau
variety, that is a variety with trivial canonical class. Such compactification is called \emph{a Calabi--Yau compactification}.
  \item[Toric condition] There is a degeneration
   $X\rightsquigarrow T$ to a toric variety~$T$ whose fan polytope
  (i.\,e. the convex hull of generators of its rays) coincides with the Newton polytope
  (i.\,e. the convex hull of the support) of $f_X$.
\end{description}
\end{definition}

The series $\sum [f_X^i]t^i$ is a period of a family of fibers of a map given by
$f_X$, see, for instance,~\cite[Proposition 2.3]{Prz08a} or~\cite[Theorem~3.2]{CCGGK12} for a proof.
Thus the period condition is a numerical expression of coincidence of constant term of regularized $I$-series and a period of the family provided by $f_X$.
Let us remind that the Laurent polynomials for which only the period condition is satisfied are called \emph{very weak Landau--Ginzburg models};
ones for which in addition a Calabi--Yau condition holds are called \emph{weak Landau--Ginzburg models}.

Toric Landau--Ginzburg models have been constructed for Fano threefolds (see~\cite{Prz13}, \cite{ILP13}, and~\cite{DHKLP}) and
complete intersections in projective spaces (\cite{ILP13}); some other partial results are also known.
In Theorem~\ref{theorem: main} we prove that Fano complete intersections in Grassmannians have very weak Landau--Ginzburg
models. We believe that these models are in fact toric ones, see Problem~\ref{problem: toric LG}.

\subsection{BCFKS models}
\label{subsection:BCFKS}
In this subsection we describe some constructions from~\cite{BCFKS97} and~\cite{BCFKS98}
(see also~\cite[B25]{EHX97})
for a complete intersection in a Grassmannian $\G(n, k+n)$, $k,n\ge 2$, in a form presented in~\cite{PSh14a}
(cf.~\cite{PSh14b}). 

We define a quiver $\QQ$
as a set of vertices
$$
\Ver(\QQ)=\{(i,j)\mid i\in [1,k], j\in [1,n]\}\cup\{(0,1), (k,n+1)\}
$$
and a set of arrows
$\Ar(\QQ)$ described as follows.
All arrows are either \emph{vertical} or \emph{horizontal}.
For any $i\in [1,k-1]$ and any $j\in [1,n]$ there is one
vertical arrow~\mbox{$\vv_{i,j}=\arrow{(i,j)}{(i+1,j)}$} that goes from the vertex $(i,j)$
down to the vertex $(i+1,j)$. For any $i\in [1,k]$ and any $j\in [1,n-1]$
there is one horizontal
arrow~\mbox{$\hh_{i,j}=\arrow{(i,j)}{(i,j+1)}$} that goes from the vertex $(i,j)$ to the right
to the vertex $(i,j+1)$.
We
also add an extra vertical arrow~\mbox{$\vv_{0,1}=\arrow{(0,1)}{(1,1)}$}
and an extra horizontal
arrow~\mbox{$\hh_{k,n}=\arrow{(k,n)}{(k,n+1)}$} to~\mbox{$\Ar(\QQ)$},
see Figure~\ref{figure:quiverG36}.

\begin{figure}[htbp]
\begin{center}
\includegraphics[width=7cm]{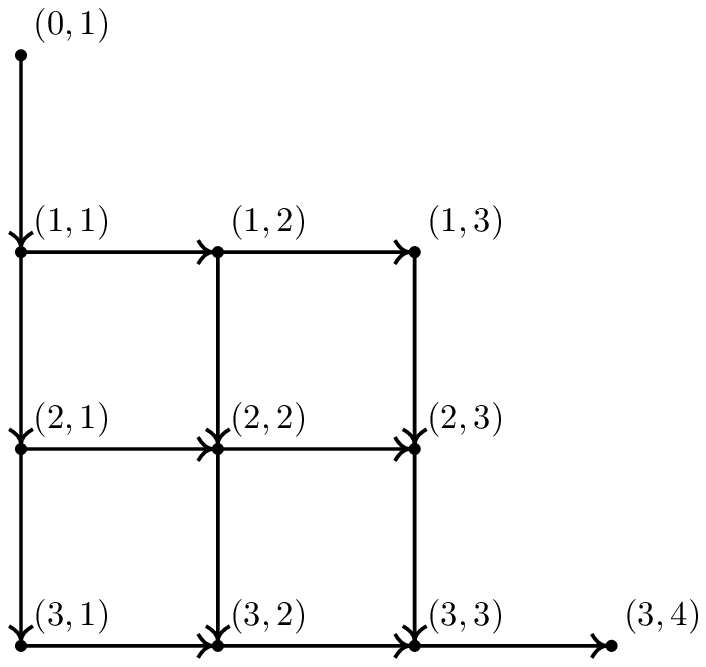}
\end{center}
\caption{Quiver $\QQ$ for Grassmannian $\G(3,6)$}
\label{figure:quiverG36}
\end{figure}

For any arrow
$$\alpha=\arrow{(i,j)}{(i',j')}\in\Ar(\QQ)$$
we define its tail $t(\alpha)$
and its head $h(\alpha)$ as the vertices $(i,j)$ and $(i',j')$, respectively.

For $r,s\in [0,k]$, $r<s$, we define a \emph{horizontal block}
$\HB(r,s)$ as a set of all vertical arrows $\vv_{i,j}$
with $i\in [r,s-1]$.
For example, the horizontal block $\HB(0,1)$ consists of a single
arrow $\vv_{0,1}$, while the horizontal block $\HB(1,3)$ consists
of all arrows $\vv_{1,j}$ and $\vv_{2,j}$, $j\in [1,n]$.
Similarly, for $r,s\in [1, n+1]$, $r<s$, we define a \emph{vertical block}
$\VB(r,s)$ as a set of all horizontal arrows $\hh_{i,j}$
with $j\in [r,s-1]$. Finally, for $r\in [0,k]$, $s\in [1,n+1]$
we define a \emph{mixed block}
$\MB(r,s)=\HB(r,k)\cup\VB(1,s)$.
For example, the mixed block $\MB(0,n)$ consists
of all arrows of $\Ar(\QQ)$ except the arrow $\hh_{k,n}$.
When we speak about a block, we mean either a horizontal,
or a vertical, or a mixed block.
We say that the \emph{size}
of a horizontal block $\HB(r,s)$ and of a vertical block $\VB(r,s)$
equals $s-r$, and the size of a mixed block $\MB(r,s)$ equals $s+k-r$.

Let $B_1,\ldots,B_l$ be blocks. We say that they are \emph{consecutive}
if the arrow $\vv_{0,1}$ is contained in $B_1$, and
for any $p\in [1,l]$ the union $B_1\cup\ldots\cup B_p$ is a block.
This happens only in one of the following
two situations: either there is an index $p_0\in [1,l]$ and sequences
of integers $0<r_1<\ldots<r_{p_0}=k$ and $0<r_1'<\ldots<r_{l-p_0}'\le n+1$
such that
\begin{multline*}B_1=\HB(0,r_1), B_2=\HB(r_1,r_2), \ldots,
B_{p_0}=\HB(r_{p_0-1}, r_{p_0}),\\
B_{p_0+1}=\VB(0, r_1'), \ldots, B_l=\VB(r_{l-p_0-1}', r_{l-p_0}'),
\end{multline*}
or there is an index $p_0\in [1,l]$ and sequences
of integers $0<r_1<\ldots<r_{p_0-1}<k$ and
$0<r_1'<\ldots<r_{l-p_0-1}'\le n+1$
such that
\begin{multline*}
B_1=\HB(0,r_1), B_2=\HB(r_1,r_2), \ldots,
B_{p_0-1}=\HB(r_{p_0-2}, r_{p_0-1}),
B_{p_0}=\MB(r_{p_0}, r_1'),\\
B_{p_0+1}=\VB(r_1',r_2'), \ldots, B_l=\VB(r_{l-p_0-2}', r_{l-p_0-1}').
\end{multline*}
The first case occurs when there are no mixed blocks among $B_1,\ldots,B_l$,
and the second case occurs when one of blocks is mixed.

Let $S=\{x_1,\ldots,x_N\}$ be a finite set. We denote the torus
$$\Spec \C [x_1^{\pm 1},\ldots, x_N^{\pm 1}]\cong (\C^*)^N$$
by $\TT(S)$. Note that $x_1,\ldots, x_N$
may be interpreted as coordinates on $\TT(S)$.

We introduce a set of variables
$\widetilde{V}=\{\widetilde{a}_{i,j}\mid i\in [1,k], j\in [1,n]\}$.
It is convenient to think that the variable $\widetilde{a}_{i,j}$
is associated to a vertex $(i,j)$ of the quiver $\QQ$.
Laurent polynomials in the variables $\widetilde{a}_{i,j}$ are
regular functions on the torus $\TT(\widetilde{V})$.
We also put $\widetilde{a}_{0,1}=\widetilde{a}_{k,n+1}=1$.

For any subset $A\subset\Ar(\QQ)$ we define
a regular function
$$\widetilde{F}_A=\sum\limits_{\alpha\in A}\frac{\widetilde{a}_{h(\alpha)}}
{\widetilde{a}_{t(\alpha)}}
$$
on the torus $\TT(\widetilde{V})$.

Let $Y$ be a complete intersection of hypersurfaces of degrees
$d_1,\ldots, d_l$ in $\G(k,n+k)$, $\sum d_i<n+k$. Consider consecutive blocks
$B_1, \ldots, B_l$ of size $d_1,\ldots, d_l$, respectively, and put
$$B_0=\Ar(\QQ)\setminus\big(B_1\cup\ldots\cup B_l\big).$$

Let $\widetilde{L}\subset\TT(\widetilde{V})$ be the subvariety
defined by equations
$$\widetilde{F}_{B_1}=\ldots=\widetilde{F}_{B_l}=1.$$
In~\cite{BCFKS97} and~\cite{BCFKS98}
it was suggested that a Landau--Ginzburg model for $Y$
is given by the variety $\widetilde{L}$ with
superpotential given by the function
$\widetilde{F}_{B_0}$.
Below we prove the following result.

\begin{theorem}\label{theorem: main}
The subvariety $\widetilde{L}$ is birational to a torus $\Y\cong(\C^*)^{nk-l}$,
and the birational equivalence $\widetilde{\tau}\colon\Y\dasharrow\widetilde{L}$
can be chosen so that
$\widetilde \tau^*\left(\widetilde F_{B_0}\right)$ is a
regular
function on~$\Y$. In particular this function is given by a Laurent polynomial.
\end{theorem}

\begin{remark}\label{remark:order}
The Laurent polynomial provided by Theorem~\ref{theorem: main}
may significantly change if one takes the degrees $d_1, \ldots, d_l$
in a different order (cf. Examples~\ref{example:1121}
and~\ref{example:1112}).
\end{remark}

To prove Theorem~\ref{theorem: main}
we will use slightly more convenient coordinates than $\widetilde{a}_{i,j}$.
Make a monomial change of variables $\psi\colon\TT(V)\to\TT(V)$ defined by
\begin{equation}
\label{eq: psi}
a_{i,j}=\widetilde{a}_{i,j}\cdot \widetilde{a}_{k,n},\quad a=\widetilde{a}_{k,n}.
\end{equation}
Put
$$V=\{a_{i,j}\mid i\in [1,k], j\in [1,n], (i,j)\neq (k,n)\}\cup \{a\}.$$
Put $a_{k,n}=1$ and $a_{0,1}=a_{k,n+1}=a$ for convenience.
As above, for any subset $A\subset\Ar(\QQ)$ we define
a regular function
$$F_A=\sum\limits_{\alpha\in A}\frac{a_{h(\alpha)}}
{a_{t(\alpha)}}
$$
on the torus $\TT(V)$. Let
$L\subset\TT(V)$ be the subvariety
defined by equations
$$F_{B_1}=\ldots=F_{B_l}=1.$$
We are going to check that
the subvariety $L$ is birational to a torus $\Y\cong(\C^*)^{nk-l}$,
and the birational equivalence $\tau\colon\Y\dasharrow L$
can be chosen so that
the pull-back of $F_{B_0}$ is a regular
function on $\Y$.
Obviously, the latter assertion is equivalent to
Theorem~\ref{theorem: main}.

\section{Proof of the main result}
\label{section: main theorem}

In this section we prove Theorem~\ref{theorem: main}.

The following assertion is easy to check.

\begin{lemma}
\label{lemma:torus}
Let $\X$ be a variety with a free action of a torus $T$. Put $\Y=\X/T$,
and let $\varphi\colon \X\to \Y$ be the natural projection.
Suppose that $\varphi$ has a section $\sigma\colon \Y\to \X$. Then one has an isomorphism
\begin{equation*}
\xi\colon \X\stackrel{\sim}\to T\times \Y.
\end{equation*}
Moreover, suppose that a function $F\in H^0(\X,\O_{\X})$
is semi-invariant with respect to the $T$-action,
i.\,e. there is a character $\chi$ of $T$
such that for any $x\in \X$ and $t\in T$ one has $F(tx)=\chi(t)F(x)$.
Then there is a function $\bar{F}\in H^0(\Y,\O_{\Y})$ such that
$F=\xi^*\big(\chi\cdot\bar{F}\big)$.
\end{lemma}
\begin{proof}
Straightforward.
\end{proof}

Recall that $B_1,\ldots,B_l$ are consecutive blocks.
In particular, the arrow $\vv_{0,1}$ is contained in~$B_1$.

We are going to define the weights $\wt_1, \ldots, \wt_l$ of the vertices
of $\QQ$ so that the following properties are satisfied.
Consider an arrow $\alpha
\in\Ar(\QQ)$.
Then
$$
\wt_p\left(h(\alpha)\right)-wt_p \left(t(\alpha)\right)=
\left\{
\begin{array}{l}
-1 \text{\ if\ } \alpha\in B_p,\\
0 \text{\ if\ } \alpha\notin B_p, \text{\ and\ } \alpha\neq\hh_{k,n}.
\end{array}
\right.
$$
Also, for any $p\in [1,l]$ we require the following properties:
\begin{itemize}
\item one has $\wt_p(i,j)\ge 0$ for all $(i,j)$;

\item one has $\wt_p(k,n)=0$, so that
$$\wt_p(k,n+1)-\wt_p(k,n)=\wt_p(k,n+1)\ge 0;$$

\item one has $\wt_p(0,1)=\wt_p(k,n+1)$.
\end{itemize}

Actually, there is only one way to assign weights
so that the above requirements are met.
Choose an index~\mbox{$p\in [1,l]$}.
If $B_p=\HB(r,s)$ is a horizontal
block, we put
$$
\wt_p(i,j)=
\left\{
\begin{array}{l}
s-i, \text{\ if\ } i\in [r,s], j\in [1,n],\\
0, \text{\ if\ } i\in [s+1,k], j\in [1,n],\\
s-r, \text{\ if\ } i\in [1,r-1], j\in [1,n], \text{\ or\ } (i,j)=(0,1).
\end{array}
\right.
$$
In particular, this gives
$\wt_p(0,1)=s-r$.
If $B_p=\MB(r,s)$ is a mixed block, we put
$$
\wt_p(i,j)=
\left\{
\begin{array}{l}
(k-i)+(s-j), \text{\ if\ } i\in [r,k], j\in [1,s],\\
k-i, \text{\ if\ } i\in [r,k], j\in [s+1,n],\\
(k-r)+(s-j), \text{\ if\ }
i\in [1,r-1], j\in [1,s], \text{\ or\ } (i,j)=(0,1),\\
k-r, \text{\ if\ } i\in [1,r-1], j\in [s+1,n].
\end{array}
\right.
$$
If $B_p=\VB(r,s)$ is a vertical block,
we put
$$
\wt_p(i,j)=
\left\{
\begin{array}{l}
s-j, \text{\ if\ } i\in [1,k], j\in [r,s],\\
s-r, \text{\ if\ } i\in [1,k], j\in [1,r-1], \text{\ or\ } (i,j)=(0,1),\\
0, \text{\ if\ } i\in [1,k], j\in [s+1,n].
\end{array}
\right.
$$
Finally, we always put $\wt_p(k,n+1)=\wt_p(0,1)$.

An example of weights assignment corresponding
to Grassmannian $\G(3,6)$ and mixed block $B=\MB(2,2)$
is given on Figure~\ref{figure:weightsG36}.
The solid arrows are ones that are
contained in $B$, while the dashed
arrows are those of $\Ar(\QQ)\setminus B$.
The weight vertex $(3,1)$ of $B$ is marked by a white circle.

\begin{figure}[htbp]
\begin{center}
\includegraphics[width=7cm]{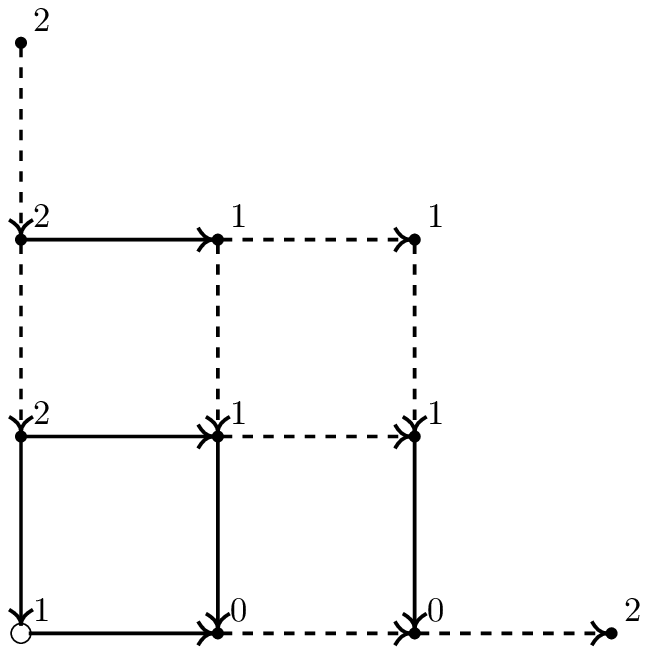}
\end{center}
\caption{Weights for Grassmannian $\G(3,6)$ and mixed block $\MB(2,2)$}
\label{figure:weightsG36}
\end{figure}

To any block $B$ we associate a \emph{weight vertex} of the quiver $\QQ$
as follows. If $B=\HB(r,s)$ is a horizontal block,
then its weight vertex is $(s-1,1)$. If $B$ is a mixed block $\MB(r,s)$ or
a vertical 
block $\VB(r,s)$, then its weight vertex is $(k,s-1)$.
If $B$ is a block and $(i,j)$ is its weight vertex, we
define the \emph{weight variable} of $B$ to be $a_{i,j}$
provided that $(i,j)\neq (0,1)$, and to be $a$ otherwise.

\begin{example}\label{example:weight-coordinates}
Consider the quiver $\QQ$ corresponding to the Grassmannian $\G(3,6)$
(see Figure~\ref{figure:quiverG36}).
Suppose that $l=4$, $B_1=\HB(0,1)$, $B_2=\HB(1,2)$, $B_3=\MB(2,2)$
and~\mbox{$B_4=\VB(2,3)$}.
Then the weight vertices of the blocks are $(0,1)$,
$(1,1)$, $(3,1)$, and $(3,2)$, respectively, and the weight variables
are $a$, $a_{1,1}$, $a_{3,1}$, and $a_{3,2}$.
\end{example}

Consider a torus
$$
\X=\TT(V)\cong(\C^*)^{nk}
$$
and a torus $T\cong (\C^*)^l$ with coordinates $w_1,\ldots,w_l$.
Define an action of $T$ on $\X$ as follows
as
$$
(w_1,\ldots, w_l)\cdot a_{i,j}=w_1^{\wt_1(i,j)}\cdot\ldots\cdot w_l^{\wt_l(i,j)}\cdot a_{i,j}
$$
for all $i\in [1,k]$, $j\in [1,n]$, $(i,j)\neq (k,n)$, and
$$
(w_1,\ldots, w_l)\cdot a=
w_1^{\wt_1(0,1)}\cdot\ldots\cdot w_l^{\wt_l(0,1)}\cdot a.
$$

Using nothing but the basic properties of weights, we obtain the following lemma.

\begin{lemma}\label{lemma:weights-Fi}
Fix $p\in [1,l]$. Then $F_{B_p}$
is a semi-invariant function on $\X$ with respect to the action of $T$
with weight $w_p^{-1}$.
\end{lemma}

Recall that
$$B_0=\Ar(\QQ)\setminus \big(B_1\cup\ldots\cup B_l\big).$$
Put $A=B_0\setminus\{\hh_{k,n}\}$.
Note that 
$F_{B_0}=F_{A}+a$.
We have

\begin{lemma}\label{lemma:weights-FA}
The function $F_{A}$ is invariant with respect to the action of $T$.
On the other hand, the function $a$ is semi-invariant with weight
$$\mu(w)=w_1^{d_1}\cdot\ldots\cdot w_l^{d_l}.$$
\end{lemma}

Consider the quotient $\Y=\X/T$,
and let $\varphi\colon \X\to \Y$ be the
natural projection.
Let $x_1,\ldots, x_l$ be weight variables
of the blocks $B_1,\ldots, B_l$, respectively, and
$\Sigma\subset \X$ be the subvariety defined by
equations
$$\{x_i=1\mid i\in [1,l]\}\subset \X.$$
Note that $T$ acts on a coordinate $x_i$ multiplying
it by
$w_i\cdot N_i$, where $N_i$ is a monomial in~\mbox{$w_{i+1},\ldots, w_l$}.
In other words, define the matrix $M$
by
$$
(w_1,\ldots,w_l)\cdot x_i=\prod w_j^{M_{i,j}} x_i.
$$
Then $M$ is integral upper-triangular 
matrix
with 
units on the diagonal.
Thus $\Sigma$ has a unique common point
with any fiber of $\varphi$. Therefore, there exists a section
$\sigma\colon \Y\to \X$ of the projection $\varphi$ whose image is $\Sigma$.
Also, we see that the action of $T$ on $\X$ is free.
By Lemma~\ref{lemma:torus} we conclude that $\X\cong T\times \Y$.
In particular, one has $\Y\cong (\C^*)^{nk-l}$.

Let $V'$ be the set of all variables of
$V$ except for $x_1,\ldots, x_l$.
We regard the variables of $V$ as coordinates on $\X$,
and the variables of $V'$ as coordinates on $\Y\cong\TT(V')$.
In these coordinates the morphism $\sigma$ is given
in a particularly simple way. Namely,
for any point $y\in \Y$ the point $\sigma(y)$
has all weight coordinates equal to $1$,
and the other coordinates equal to the corresponding coordinates of $y$.

\begin{example}\label{example:pi-sigma}
In the notation of Example~\ref{example:weight-coordinates}
one has
$$
\X=\TT\big(\{a, a_{1,1}, a_{1,2}, a_{1,3}, a_{2,1}, a_{2,2}, a_{2,3}, a_{3,1}, a_{3,2}\}\big)
$$
and
$$
\Y=\TT\big(\{a_{1,2}, a_{1,3}, a_{2,1}, a_{2,2}, a_{2,3}\}\big).
$$
The action of the torus $T\cong (\C^*)^4$ is defined
by the matrix
$$
M=\left(
    \begin{array}{cccc}
      1 & 1 & 2 & 1 \\
      0 & 1 & 2 & 1 \\
      0 & 0 & 1 & 1 \\
      0 & 0 & 0 & 1 \\
    \end{array}
  \right)
$$
as
\begin{multline*}
(w_1,w_2,w_3,w_4)\colon \big(a, a_{1,1}, a_{1,2}, a_{1,3}, a_{2,1}, a_{2,2}, a_{2,3}, a_{3,1}, a_{3,2}\big)
\mapsto\\ \mapsto
\big(w_1w_2w_3^2w_4\cdot a, w_2w_3^2w_4\cdot a_{1,1},
w_2w_3w_4\cdot a_{1,2},
w_2w_3\cdot a_{1,3},\\
w_3^2w_4\cdot a_{2,1},
w_3w_4\cdot a_{2,2}, w_3\cdot a_{2,3}, w_3w_4\cdot a_{3,1},
w_4\cdot a_{3,2}\big).
\end{multline*}
(Note that the weights corresponding to the block $B_3$ can be seen on Figure~\ref{figure:weightsG36}.)
The matrix
$$
M^{-1}=\left(
    \begin{array}{rrrr}
      1 & -1 & 0 & 0 \\
      0 & 1 & -2 & 1 \\
      0 & 0 & 1 & -1 \\
      0 & 0 & 0 & 1 \\
    \end{array}
  \right)
$$
gives $w_1^{-1}=\frac{a}{a_{1,1}}$, $w_2^{-1}=\frac{a_{1,1}a_{3,2}}{a_{3,1}^2}$, $w_3^{-1}=\frac{a_{3,1}}{a_{3,2}}$, and $w_4^{-1}=a_{3,2}$, so
the projection $\varphi\colon \X\to \Y$ is given by
\begin{multline*}
\varphi\colon \left(a, a_{1,1}, a_{1,2}, a_{1,3}, a_{2,1}, a_{2,2}, a_{2,3}, a_{3,1}, a_{3,2}\right)
\mapsto\\ \mapsto
\left( \frac{a_{3,1}}{a_{1,1}a_{3,2}}\cdot a_{1,2},
\frac{a_{3,1}}{a_{1,1}}\cdot a_{1,3},
\frac{a_{3,2}}{a_{3,1}^2}\cdot a_{2,1},
\frac{1}{a_{3,1}}\cdot a_{2,2}, \frac{a_{3,2}}{a_{3,1}}\cdot a_{2,3} \right),
\end{multline*}
and the section $\sigma\colon \Y\to \X$ is given by
$$
\sigma\colon \big(a_{1,2}, a_{1,3}, a_{2,1}, a_{2,2}, a_{2,3}\big)\mapsto
\big(1, 1, a_{1,2}, a_{1,3}, a_{2,1}, a_{2,2}, a_{2,3}, 1, 1\big).
$$
\end{example}

Applying Lemma~\ref{lemma:weights-Fi} together with
Lemma~\ref{lemma:torus}, we see that there exist regular functions
$\bar{F}_p$, $p\in [1,l]$, on $\Y$ such that
under the identification $\X\cong T\times \Y$ one has
$$F_p=w_p^{-1}\cdot \varphi^*\bar{F}_p.$$
Similarly, applying Lemma~\ref{lemma:weights-FA} together with
Lemma~\ref{lemma:torus}, we see that there exist regular functions
$\bar{F}_{A}$ and $\bar{a}$ on $\Y$ such that
$F_{A}=\varphi^*\bar{F}_{A}$ and $a=\mu(w)\varphi^*\bar{a}$.

Consider a rational map
$$
y\mapsto \big(\bar{F_1}(y),\cdots,\bar{F}_l(y)\big)
$$
from $\Y$ to $T$.
Define a rational map $\tau\colon \Y\dasharrow \X$
as
$$
y\mapsto \big(\bar{F_1}(y),\cdots,\bar{F}_l(y)\big)\cdot
\sigma(y).
$$
It is easy to see that
the closure of the image of $\Y$ under the map $\tau$ is the
subvariety~\mbox{$L\subset \X$}.
In particular, $\tau$ gives a birational equivalence between
$\Y$ and $L$.

Now it remains to notice that
$$\tau^*F_{A}=\tau^*\varphi^*\bar{F}_A=\bar{F}_A.$$
On the other hand,
one has
$$\tau^*a=
\mu\big(\bar{F_1}(y),\cdots,\bar{F}_l(y)\big)\sigma^*\varphi^*\bar{a}=
\mu\big(\bar{F_1}(y),\cdots,\bar{F}_l(y)\big)\bar{a}.
$$
This means that the map
$\widetilde \tau=\tau\varphi\psi$, where $\psi$ is defined by~\eqref{eq: psi}, provides a birational map
required for Theorem~\ref{theorem: main}.

\begin{remark}
\label{remark: algorithm}
The above proof of Theorem~\ref{theorem: main} provides a very explicit
way to write down the Laurent polynomial $\tau^*F_{B_0}$. Namely,
consider a complete intersection~\mbox{$Y\subset \G(n,n+k)$}
of hypersurfaces of degrees $d_i$, $i\in [1,l]$.
The following cases may occur.
\begin{itemize}
  \item One has $d_1+\ldots+d_l\le k.$
Put $u_i=d_1+\ldots+d_i$ for $i\in [1,l]$.
Then the BCFKS Landau--Ginzburg model for $Y$ is birational to $(\CC^*)^{nk-l}$ with superpotential
$$
\sum_{i=u_{l}+1}^{k}\sum_{j=1}^n \frac{a_{i,j}}{a_{i-1,j}}+\sum_{i=1}^{k}\sum_{j=2}^n \frac{a_{i,j}}{a_{i,j-1}}
+a\left(\frac{a_{1,1}}{a}+\sum_{i=2}^{d_1}\sum_{j=1}^n \frac{a_{i,j}}{a_{i-1,j}}\right)^{d_1}
\prod_{p=2}^l
\left(\sum_{i=u_{p-1}}^{u_p}\sum_{j=1}^n \frac{a_{i,j}}{a_{i-1,j}}\right)^{d_p},
$$
where we put $a_{1,u_1-1}=1$ if $u_1>1$ and $a=1$ otherwise,
$a_{1,u_i-1}=1$ for $i\in[2,l]$, and $a_{k,n}=1$.

  \item One has $d_1+\ldots+d_l> k.$ Let $m\in [0,l-1]$
be the maximal index such that
$d_1+\ldots+d_m\le k.$
Put $u_i=d_1+\ldots+d_i$ for $i\in [1,m]$ and
$u_i=d_1+\ldots+d_i-k$ for $i\in [m+1,l]$.

If $m=0$, then the BCFKS Landau--Ginzburg model for $Y$ is birational to $(\CC^*)^{nk-l}$ with superpotential
\begin{multline*}
\sum_{i=1}^{k}\sum_{j=u_l+1}^n \frac{a_{i,j}}{a_{i,j-1}}+\\
+a\left(\frac{a_{1,1}}{a}+
\sum_{i=2}^{k}\sum_{j=1}^n \frac{a_{i,j}}{a_{i-1,j}}+
\sum_{i=1}^{k}\sum_{j=2}^{u_1} \frac{a_{i,j}}{a_{i,j-1}}\right)^{d_1}
\cdot \prod_{p=2}^{l}\left(\sum_{i=1}^k\sum_{j=u_{p-1}}^{u_{p}}
\frac{a_{i,j}}{a_{i,j-1}}\right)^{d_p}.
\end{multline*}
If $m>1$, then the BCFKS Landau--Ginzburg model for $Y$ is birational to $(\CC^*)^{nk-l}$ with superpotential
\begin{multline*}
\sum_{i=1}^{k}\sum_{j=u_l+1}^n \frac{a_{i,j}}{a_{i,j-1}}
+a\left(\frac{a_{1,1}}{a}+\sum_{i=2}^{d_1}\sum_{j=1}^n \frac{a_{i,j}}{a_{i-1,j}}\right)^{d_1}\cdot\\
\cdot \prod_{p=2}^m
\left(\sum_{i=u_{p-1}}^{u_p}\sum_{j=1}^n
\frac{a_{i,j}}{a_{i-1,j}}\right)^{d_p}\cdot
\left(\sum_{i=u_m}^{k}\sum_{j=1}^n \frac{a_{i,j}}{a_{i-1,j}}+\sum_{i=1}^k\sum_{j=2}^{u_{m+1}} \frac{a_{i,j}}{a_{i,j-1}}\right)^{d_{m+1}}\cdot \\ \cdot
\prod_{p=m+2}^{l}\left(\sum_{i=1}^{k}\sum_{j=u_{p-1}}^{u_p} \frac{a_{i,j}}{a_{i,j-1}}\right)^{d_p}
,
\end{multline*}

In both cases we put $a_{1,u_1-1}=1$ if $u_1>1$ and $a=1$ otherwise,
$a_{1,u_p-1}=1$ for $p\in[2,m]$, $a_{k,u_p-1}=1$ for $p\in [m+1,l]$, and $a_{k,n}=1$.

\end{itemize}

\end{remark}

\section{Givental's integral}
\label{section: integral}

We discuss now period integrals for Laurent polynomials given by Theorem~\ref{theorem: main}.
For more details see~\cite{PSh14a}.

Given 
an ordered
set of variables
$Z=\{z_1,\ldots,z_r\}$,
define a \emph{standard logarithmic form} as the form
\begin{equation*}
\label{eq: standard form}
\Omega(z_1,\ldots,z_r)=\Omega(Z)=\frac{1}{(2\pi \sqrt{-1})^r}\frac{dz_1}{z_1}\wedge\ldots\wedge\frac{dz_r}{z_r}.
\end{equation*}

Consider an integral
$$
\int_{\sigma} \frac{du}{u}\wedge \Omega_0
$$
for some form $\Omega_0$ dependent on the variable $u$. Let
$$\sigma=\sigma'\cap \{|u|=\varepsilon\}$$
for some cycle $\sigma'$.

It is well known that
(see, for instance,~\cite[Theorem 1.1]{ATY85}) that
$$
\frac{1}{2\pi \sqrt{-1}}\int_{\sigma} \frac{du}{u}\wedge \Omega_0=\int_{\sigma'}\left.\Omega_0\right|_{u=0}
$$
if both integrals are well defined (in particular the form $\Omega_0$ does not have a pole along~\mbox{$\{u=0\}$}).



Consider a Fano complete intersection $Y\subset G=\G(n,n+k)$ of hypersurfaces of degrees $d_i$, $i\in [1,l]$. Let
$F_i=F_{B_i}$
be polynomials defining BCFKS Landau--Ginzburg model for $Y$ and let
$F_0=F_{B_0}$ be the superpotential.

Given a torus $\TT(\{x_1,\ldots,x_r\})$ we call a cycle $\{|x_i|=\varepsilon_{i}\mid i\in [1,r]\}$ depending on some
real numbers $\varepsilon_i$
\emph{standard}. 

\begin{definition}[see~\cite{BCFKS97}]
\label{def:integral}
\emph{An (anticanonical) Givental's integral} for $Y$ is an integral
\begin{equation*}
\label{eq:restricted integral}
I_Y^0=\int_{\delta}
\frac{
\Omega (\{\widetilde a_{i,j}\})
}{\prod_{j=1}^l\left(1-\widetilde F_j\right)\cdot \left(1-t\widetilde F_0\right)} \in \CC[[t]]
\end{equation*}
for a standard cycle $\delta=\{|\widetilde a_{i,j}|=\varepsilon_{i,j}\mid i\in [1,k], j\in [1,n], \varepsilon_{i,j}\in \R_+\}$,
whose orientation 
is chosen such that $\left.I_Y^0\right|_{t=0}=1$.
\end{definition}

\begin{remark}
The integral $I_Y^0$ does not depend on numbers $\varepsilon_{i,j}$ provided they are small enough.
\end{remark}


%
%
In~$\mbox{\cite[Conjecture 5.2.3]{BCFKS97}}$ it is conjectured
that $\widetilde{I}^G_0=I_G^0$, and a formula for~$\widetilde{I}^G_0$ is provided.
This conjecture was
proved for $n=2$ in~\cite[Proposition 3.5]{BCK03} and for any~$n\ge 2$
in~\cite{MR13}.
In discussion after Conjecture~5.2.1 in~\cite{BCFKS98} it is explained
that from the latter theorems and the Quantum Lefschetz Theorem it follows that Givental's integral $I^0_{Y}$ equals $\widetilde{I}^Y_0$.
We summarize the results mentioned above as follows.

\begin{theorem}
\label{theorem: periods of CI}
Let $Y=\G(n,k+n)\cap Y_1\cap \ldots \cap Y_l$
be a smooth Fano complete intersection. Denote $d_i= \deg Y_i$ and $d_0=k+n-\sum d_i$.
Then 
$$
\widetilde{I}_0^Y=I^0_Y=\sum_{d\ge 0}\sum_{s_{i,j}\ge 0} \frac{\prod_{i=0}^l (d_id)!}{(d!)^{k+n}} \prod_{i=1}^{k-1} \prod_{j=1}^{n-1} \binom{s_{i+1,j}}{s_{i,j}} \binom{s_{i,j+1}}{s_{i,j}} t^{d_0d},
$$
where we put $s_{k,j}=s_{i,n}=d$. 
\end{theorem}

It turns out that changes of variables constructed in Theorem~\ref{theorem: main} preserve this period.

\begin{proposition}
\label{proposition:periods}
The period condition holds for Laurent polynomials given by Theorem~\ref{theorem: main}.
In other words, Theorem~\ref{theorem: main} provides very weak Landau--Ginzburg models for Fano complete intersections
in Grassmannians. 
\end{proposition}

\begin{proof}
We follow the notation from Theorem~\ref{theorem: main}.
A toric change of variables $\varphi\psi$ change coordinates $\{\widetilde a_{i,j}\}$ by coordinates $\{w_i\}\cup V'$.
One gets
\begin{multline*}
I_Y^0=\int_{\delta}
\frac{
\Omega (\{\widetilde a_{i,j}\})
}{\prod_{j=1}^l\left(1-\widetilde F_j\right)\cdot \left(1-t\widetilde F_0\right)} =\\=
\int_{\delta'} \Omega(V')
\wedge\left(\bigwedge_{j=1}^l\left(\frac{1}{2\pi \sqrt{-1}}\frac{dw_j}{w_j\cdot\left(1-\bar F_{j}/w_j\right)}\right)\right)\cdot \frac{1}{1-t
\bar F}
\end{multline*}
for an appropriate choice of an orientation on $\delta'$, where $\bar F= \bar{F}_{A}+\mu(w)\cdot \bar{a}$.
Following the birational isomorphism $\tau$, consider variables $u_i=w_i-\bar F_{i}$ instead of $w_i$.
Then, after appropriate choice of cycle $\Delta'$ (cf.~\cite[proof of Proposition 10.5]{PSh14a}) one gets
\begin{multline*}
I_Y^0=
\int_{\delta'} \Omega(V')
\wedge \left(\bigwedge_{j=1}^l\left(\frac{1}{2\pi \sqrt{-1}}\frac{dw_j}{w_j-\bar F_{j}}\right)\right)\cdot \frac{1}{1-t \bar F}=\\=
\int_{\Delta'} \Omega(V')
\wedge\left(\bigwedge_{j=1}^l\left(\frac{1}{2\pi \sqrt{-1}}\frac{du_j}{u_j}\right)\right)\cdot \frac{1}{1-t F_u}=
\int_{\Delta} \frac{
\Omega(V')
}{1-t f}=\sum [f^i]t^i,
\end{multline*}
where $\Delta$ is a projection of $\Delta'$ on $\TT(V)$ and $F_u$ is a result of replacement of $w_i$ by
of $u_i+F_{B_i}$ in $\bar F$.
\end{proof}

Fano complete intersections in projective spaces are proven to have toric Landau--Ginzburg models, see~\cite{Prz13} and~\cite{ILP13}.
In other words, it is proven that very weak Landau--Ginzburg models for complete intersections have Calabi--Yau compactifications,
and these very weak Landau--Ginzburg models correspond to toric degenerations of complete intersections.
Using the Calabi--Yau compactifications in~\cite{PSh15} a particular Hodge number of the complete intersections is computed in terms
of Landau--Ginzburg models.
The natural problem is to extend these results to complete intersections in Grassmannians.

\begin{problem}[{cf.~\cite[Problem 17]{Prz13}}]
\label{problem: toric LG}
Let $Y$ be a Fano complete intersection in $\G(n, k+n)$, and let $f_Y$ be the Laurent polynomial for $Y$ given by Theorem~\ref{theorem: main}. Prove that $f_Y$ is a toric Landau--Ginzburg model. Prove that the number of components
of a central fiber of a Calabi--Yau compactification for $f_Y$ is equal to $h^{1,\mathrm{dim}\,Y-1}(Y)+1$.
\end{problem}


\begin{example}\label{example:1121}
Let $Y$ be a smooth
intersection of the Grassmannian $\G(3,6)$ with
a quadric and three hyperplanes.
Put $l=4$, $d_1=d_2=d_4=1$, and $d_3=2$.
The BCFKS Landau--Ginzburg model in this case
is birational to a torus
$$
\Y\cong \TT\big(\{a_{1,2}, a_{1,3}, a_{2,1}, a_{2,2}, a_{2,3}\}\big)
$$
with a superpotential
$$
f_Y=\left(a_{2,1}+\frac{a_{2,2}}{a_{1,2}}+\frac{a_{2,3}}{a_{1,3}}\right)\cdot
\left(\frac{1}{a_{2,1}}+\frac{a_{3,2}}{a_{2,2}}+
\frac{1}{a_{2,3}}+a_{1,2}+\frac{a_{2,2}}{a_{2,1}}+1\right)^2\cdot
\left(\frac{a_{1,3}}{a_{1,2}}+\frac{a_{2,3}}{a_{2,2}}+1\right)
$$
given by Remark~\ref{remark: algorithm}.
By
Theorem~\ref{theorem: periods of CI} (see also~\cite[Example~5.2.2]{BCFKS97}) one has
\begin{multline}\label{eq:const-term-1121}
I^0_Y=\sum_{d, b_1,b_2,b_3,b_4}\frac{(2d)!}{(d!)^{2}} \binom{b_2}{b_1}  \binom{b_3}{b_1}  \binom{d}{b_2}  \binom{b_4}{b_2}
 \binom{b_4}{b_3}  \binom{d}{b_3}  \binom{d}{b_4}^2t^{d}=\\
 =1+12t+756t^2+78960t^3+10451700t^4+1587790512t^5+263964176784t^6+\\
 +46763681545152t^7+8685492699286260t^8+\cdots.
\end{multline}
One can check that
the first few terms we write down on the right
hand side of~\eqref{eq:const-term-1121}
equal the first few terms of the series~\mbox{$\sum [f_Y^i]t^i$}.
\end{example}

\begin{example}\label{example:1112}
Let $Y$ be a smooth
intersection of the Grassmannian $\G(3,6)$ with
a quadric and three hyperplanes, i.\,e. the variety
that was already considered in Example~\ref{example:1121}.

Put $l=4$, $d_1=d_2=d_3=1$, and $d_4=2$.
One has
$$
\Y=\TT\big(\{a_{1,2}, a_{1,3}, a_{2,2}, a_{2,3}, a_{3,1}\}\big).
$$
By Remark~\ref{remark: algorithm} we get
$$
f_Y=\left(1+\frac{a_{2,2}}{a_{1,2}}+\frac{a_{2,3}}{a_{1,3}}\right)\cdot
\left(a_{3,1}+\frac{1}{a_{2,2}}+\frac{1}{a_{2,3}}\right)\cdot
\left(a_{1,2}+a_{2,2}+\frac{1}{a_{3,1}}+\frac{a_{1,3}}{a_{1,2}}
+\frac{a_{2,3}}{a_{2,2}}+1\right)^2.
$$
One can check that the first few constant terms $[f_Y^i]$
coincide with the first few terms of the series presented on the right
side of~\eqref{eq:const-term-1121}.
Note that the Laurent polynomial
$f_Y$ can't be obtained from the polynomial from Example~\ref{example:1121}
by monomial change of variables (cf. Remark~\ref{remark:order}).
It would be interesting to find out if
these two Laurent polynomials
are mutational equivalent (cf.~\cite[Theorem~2.24]{DH15}).
\end{example}

\end{document}